\documentclass[11pt]{amsart}

\newtheorem{proposition}{Proposition}
\newtheorem{theorem}{Theorem}
\newtheorem{corollary}{Corollary}

\newtheorem{definition}{Definition}

\usepackage{amssymb}
\usepackage{amsthm}

\newcommand{\set}[1]{\{#1\}}

\newcommand{\map}{\longrightarrow}
\newcommand{\transpose}{\kern1pt{}^t\kern-0.5pt}

\newcommand{\BB}{{\mathbb B}}
\newcommand{\CC}{{\mathbb C}}

\newcommand{\HH}{{\mathbb H}}

\newcommand{\PP}{{\mathbb P}}
\newcommand{\QQ}{{\mathbb Q}}
\newcommand{\RR}{{\mathbb R}}

\newcommand{\ZZ}{{\mathbb Z}}

\newcommand{\AAA}{{\mathcal A}}

\newcommand{\EEE}{{\mathcal E}}

\newcommand{\HHH}{{\mathcal H}}

\newcommand{\MMM}{{\mathcal M}}

\newcommand{\OOO}{{\mathcal O}}
\newcommand{\PPP}{{\mathcal P}}

\title{Cubic Surfaces with Special Periods}
\author{James Carlson and Domingo Toledo}

\date{April 9, 2011, revised October 2, 2011. Research partially supported by National Science 
Foundation Grant DMS-0600816;
 the first author also gratefully acknowledges the support of the Clay
 Mathematics Institute and of CIMAT.  \\ 
 } 
                                           
\begin{document}
\maketitle


\parskip=1pt
\tableofcontents

\parskip=5pt

\section{Introduction}

Consider a complex elliptic curve $E$ with equation 
\begin{equation}
  \label{equation:weierstrass}
  y^2 =  4x^3 - g_2x -g_3.
\end{equation}
The endomorphism ring of such a curve always contains the integer 
dilations $z \mapsto nz$.  An elliptic curve with a larger endomorphism 
ring is said to have  \emph{complex multiplication}.  
The additional endomorphisms are of the form $z 
\mapsto \lambda z$, and it is easy to see that $\lambda$ lies in a 
purely 
imaginary quadratic extension $K = 
\QQ(\sqrt{-d})$ of 
the rational numbers.  This number field can be 
identified with $End(E)\otimes \QQ$ and is called \emph{the CM field} of $E$.  Conversely, 
if $K$ is a purely imaginary quadratic extension of the rational numbers,
an elliptic curve with CM field $K$ can be constructed as $\CC/\OOO_K$,
where $\OOO_K$ is the ring of integers in $K$.

Now fix a symplectic homology basis $\set{\gamma_1, \gamma_2}$, and let
\[
    \lambda_i = \int_{\gamma_i} \frac{dx}{y}
\]
be the \emph{fundamental periods}. If we write $(\ref{equation:weierstrass})$ as
\[
   y^2 = 4(x - e_1)(x - e_2)(x - e_3),
\]
then fundamental periods are the elliptic integrals
\[
   \lambda_i = \int_{e_i}^{e_{i+1}} \frac{dx}{\sqrt{(x - e_1)(x - e_2)(x - e_3)}}, \ \text{for}\ i = 1,2.
\]
A theorem of Siegel (1932)
tells us that if the coefficients $g_i$ are algebraic numbers, then at least one of the periods is transcendental, and in 1934,
Schneider showed that the nonzero periods are always transcendental.   For a specific
example, consider the Fermat elliptic curve, with affine equation $y^2  = 4(x^3  - 1)$. 
Then with $e_1,e_2,e_3= 1,\omega,\omega^2$, in that order,
\[
   \lambda_2 =  \int_1^\infty \frac{dx}{\sqrt{x^3  - 1}} = \frac{1}{3}B(1/6,1/2)
      = \frac{\Gamma(1/3)^3}{2^{4/3}\pi}.
\]
See \cite[equation (10)]{Waldschmitt}.    Here, and throughout this paper, $\omega= e^{2\pi i/3}$, a primitive cube root of unity.

When $E$ has complex multiplication,  the \emph{period ratio}
\[
    \tau = \lambda_2/\lambda_1.
\]
lies in the CM field.  In this case, the transcendence degree of the field the periods, 
$\QQ(\lambda_1, \lambda_2)$, is one. It is conjectured that in the non-CM case, the field generated by
the periods is of transcendence degree two. 

Consider next a cubic surface $S$. Because the  Hodge structure on $H^2(S)$ is entirely of type $(1,1)$, there are no periods of interest.  However, there is an auxiliary Hodge structure of weight three associated to $S$ which does have interesting 
periods and which behaves in many ways like the Hodge 
structure of an elliptic curve. This is the Hodge structure of a three-sheeted cover $T$ of projective 3-space branched along $S$.
The period domain for such Hodge structures is $\BB^4$,  the 
unit ball in complex four-space, which is the same as the complex hyperbolic space $\CC H^4$.  In this paper we use the notation $\BB^4$ rather than $\CC H^4$.

  There is a natural period map
\begin{equation}
\label{eq:periodmap}
   \PPP: \MMM_{st} \map \Gamma\backslash\BB^4
\end{equation}
from the moduli space of stable 
cubic surfaces (those with at worst nodal singularities)
to the indicated quotient of the ball by an arithmetic group.  
It is an isomorphism \cite{ACT}.  The aim of this paper is to 
explore the relation between cubic surfaces and their periods.
The theory of elliptic curves serves as a model of what might
be possible to establish, at least in part.  In particular, we 
will investigate the rationality of period ratios.  While the 
results presented in this direction are modest, they point to 
several interesting questions.  We discuss them at the end of 
this paper.

We are indebted to Madhav Nori, who explained to us how
to use the different ideal to check whether the Abelian variety defined
by a number field is principally polarized.   We would also like to express our
thanks to the referee, whose comments greatly improved this paper.

\section{Statement of results}
\label{section:results}

We first  review the main results of \cite{ACT} and explain the mapping (\ref{eq:periodmap}).  We  summarize the results we need,  details can be found in \cite{ACT} and in \cite{Beau}.  

\subsection{Standard Notation}
\label{subsec:notation} 

The following notation will be used throughout this paper.  
If  $S\subset\PP^3$ is a cubic surface defined by a cubic equation $F(x) = 0$, let $T$ be the cubic threefold given by $\{y^3 = F(x)\}\subset\PP^4$,  where $x\in\CC^4$ and $y\in\CC$.   Let  $\omega$  denote the primitive cube root of unity $\omega = \exp( 2\pi i/3)$,  and let  $\sigma$ be  the  automorphism of  $T$ over $\PP^3$ defined by $\sigma(x,y) = (x,\omega y)$,   Finally, let $\theta = i\sqrt{3}$ be the square root of $-3$ situated on the positive imaginary axis in the complex plane.

\subsection{The Eisenstein structure of $T$}
\label{subsec;eisensteinstructure}

The middle cohomology  $H^3(T,\ZZ)$ is a free abelian group of rank $10$ with an automorphism $\sigma$ satisfying $\sigma^2 + \sigma + 1 =0$, hence it is a free module of rank $5$ over the ring $\EEE = \ZZ[\omega]$ of Eisenstein integers. 
This  
module carries a natural $\EEE$-valued hermitian  form $h$ of signature 
$(4,1)$ and determinant one given by  
\[
  -  2h(x,y) = \left<(\sigma - \sigma^{-1})x,y\right> + (\omega - 
\omega^{-1})
\left<x,y\right>,
\]
see (2.3.1) of \cite{ACT}, or, more simply, by the equivalent formula 
\[
h(x,y) = <x,\sigma y> - \omega <x,y>
\]
as in (2.1) of \cite{Beau}.  This form is standard:  $(H^3(T,\ZZ),\sigma,h)$ is isomorphic to $(\EEE^5,\omega,h_0)$ where $h_0(z,w) = -z_0\overline{w}_0 + z_1\overline{w}_1 + \cdots + z_4\overline{w}_4$.

\subsection{The Hodge structure of $T$}
\label{subsec:hodgestructure}

 Since $H^{3,0}(T) = 0$, 
the Hodge structure has the form $H^{2,1} \oplus  H^{1,2}$. 
The symmetry $\sigma$ is an automorphism of $H^3 = H^3(T,\CC)$ with 
eigenvalues $
\omega$
and $\bar\omega$.  Let $H^3 = H^3_\omega \oplus H^3_{\bar\omega}$ be the 
eigenspace 
decomposition.  The summands are in natural duality by the cup-product pairing and are interchanged by complex conjugation.  
Because $\sigma$ is holomorphic, its action on cohomology is compatible with the Hodge decomposition.  Thus one has  
complex Hodge 
structures
\[
    H^{2,1}_{\omega} \oplus H^{1,2}_{\omega}\ \text{on} \ H^3_{\omega}\ \ \ \text{and}\ \ \  H^{2,1}_{\bar\omega} \oplus H^{1,2}_{\bar\omega}\ \text{on} \ H^3_{\bar\omega},
\]
where the summands have dimensions $4,1,1,4$ in the order listed.   Equivalently, the same summands  give the eigenspace decompositions
\[
 H^{2,1} = H^{2,1}_\omega\oplus H^{2,1}_{\bar\omega}\ \ \text{and}\ \ H^{1,2} = H^{1,2}_\omega\oplus H^{1,2}_{\bar\omega}.
\]
The hermitian form
\begin{equation}
\label{eq:hermitian}
h'(\alpha,\beta) = - i \sqrt{3} <\alpha,\bar\beta>
\end{equation}
on $H^3(T,\CC)$ is positive definite on $H^{2,1}$, negative definite on $H^{1,2}$ and these two subspaces are $h'$-orthogonal.    There are complex linear isomorphisms
\begin{equation}
\label{eq:frame}
(H^3_\omega,h') \cong( ( H^3(T,\ZZ),\sigma)\otimes_\EEE \CC, h)\cong\EEE^{4,1}\otimes_\EEE\CC\cong\CC^{4,1}
\end{equation}
and
\begin{equation}
\label{eq:framing}
 (H^3_{\bar\omega},h') \cong( ( H^3(T,\ZZ),\sigma^{-1})\otimes_\EEE \CC,-\bar h)\cong\EEE^{1,4}\otimes_\EEE\CC\cong\CC^{1,4},
\end{equation}
where $\EEE^{4,1}$ is the standard model $(\EEE^5,\omega,h_0)$ above, and $\EEE^{1,4}$ is $(\EEE^5,\omega,-\bar h_0)$. 
Thus the decomposition $H^3(T) = H^3_\omega\oplus H^3_{\bar\omega}$ is a decomposition 
\begin{equation}
\label{eq:decomposition}
\CC^{5,5} = \CC^{4,1}\oplus\CC^{1,4},
\end{equation}
and the integral sub-module $\ZZ^{10}\subset \CC^{10}$ corresponding to  $H^3(T,\ZZ)\subset H^3(T,\CC)$ projects to Eisenstein \emph{lattices} $\EEE^{4,1}$, $\EEE^{1,4}$ in each summand.

The Hodge structure on $H^3(T)$ is determined by either subspace $H^{2,1}_\omega\subset H^3_\omega$ or $H^{2,1}_{\bar\omega}\subset H^3_{\bar\omega}$, since each  subspace determines the other by conjugation and orthogonality.   Either collection constitutes a ball $\BB^4$ in a Grassmannian, namely the positive hyperplanes in $\CC^{4,1}$, respectively the positive lines in $\CC^{1,4}$.   We choose the second description for the $\BB^4$ of (\ref{eq:periodmap}) with the consequent description of $\Gamma = PU(1,4,\EEE)$, the projectivized group of $\EEE$-linear isometries of $\EEE^{1,4}$.

\subsection{The period map}
\label{subsec:periodmap}

The \emph{period map} $\PPP$ is defined on a covering space of the space of cubic forms.  The points of this covering space corresponding to a way of choosing the isomorphisms in (\ref{eq:framing}), which is called a \emph{framing} in (3.9) of \cite{ACT}.   If $\phi$ is the chosen isomorphism, $\PPP(F,\phi) = \phi(H^{2,1}_\omega)\subset\CC^{1,4}$.  It descends to the map $\PPP$ of (\ref{eq:periodmap}).

Note that the choice of hermitian form in (2.4) of \cite{ACT} is the negative of the standard choice (\ref{eq:hermitian}). Consequently, we have to make all the obvious sign changes when quoting \cite{ACT}.  Note also that the isomorphisms  (\ref{eq:frame}) and (\ref{eq:framing}) are complex linear, and incorporate  the correction added in proof to (2.2) of \cite{ACT}.   Our choice is equivalent to Beauville's choice of a positive hyperplane in $\CC^{4,1}$, and $\PPP$ is holomorphic.

\subsection{The intermediate Jacobian}
\label{subsec:intermediatejacobian}

This period map determines $J(T)$, the intermediate Jacobian of $T$, which is a five-dimensional principally polarized abelian variety, thus classified by a point in the moduli space $\AAA_5 = Sp(10,\ZZ)\backslash \HH_5$, where $\HH_5$ is the Siegel upper half plane of genus $5$.  Taking a standard model $(\ZZ^{10},<\ ,\ >)$ of a free abelian group of rank $10$ with integral symplectic form $<\ ,\ >$ and defining hermitan form $h'$ on $\ZZ^{10}\otimes\CC = \CC^{10} = \CC^{5,5}$ as above,  $\HH_5$ is given as the space of positive Lagrangian subspaces.  These subspaces of $\CC^{10}$ are $5$-dimensional, $< \ ,\ >$-isotropic, and $h'$-positive.  We obtain a map $\iota:\BB^4\map\HH_5$ defined by
\[
\iota (l) = \bar l^\bot\oplus l \subset \CC^{5,5} = \CC^{4,1}\oplus\CC^{1,4},
\]
see (9.2.1) of \cite{ACT} and (4.2) of \cite{Beau}.  This gives a holomorphic, totally geodesic embedding $\iota:\BB^4\map \HH_5$ which is equivariant with respect to the natural homomorphism $\lambda:\Gamma\to Sp(10,\ZZ)$ given by $\lambda(\gamma) = \bar\gamma\oplus \gamma = {}^t\gamma^{-1}\oplus \gamma$ in the decomposition (\ref{eq:decomposition}) into spaces that are simultaneously dual and conjugate to each other (transposition is with respect to the complex  bilinear form $h_0(z,\bar w)$).   The image of $\lambda$ is the centralizer of $\sigma$ in $Sp(10,\ZZ)$.   It follows that  $\iota$  descends to an embedding $\iota:\Gamma\backslash\BB^4\map \AAA_5$, and the image of $\iota$ consists of those principally polarized abelian varieties that are isomorphic to one on which $\sigma$ acts.

\subsection{The hyperplane arrangement}
\label{subsec:hyperplanes}
The group $\Gamma$ is generated by complex reflections of order $6$ (hexaflections) on the hyperplanes $v^\bot$, where $v\in\EEE^{1,4}$ and $h_0^*(v) =-1$,  where $h_0^* = -\bar h_0$ is the standard form on $\CC^{1,4}$, see \S 7 of \cite{ACT}.   Let $\HHH\subset\BB^4$ be defined by 
\[
\HHH =\cup\{\BB^3(v^\bot):v\in\EEE^{1,4},\  h_0^*(v) =  - 1\}
\]
where $\BB^3(v^\bot)\subset\BB^4$ is the ball of positive lines of the $(1,3)$-form $h_0^*|v^\bot$.   This is a locally finite collection of hyperplanes and the group $\Gamma$ acts transitively on its elements.    
Moreover,  $\HHH$  is an \emph{orthogonal arrangement} of hyperplanes, meaning that any two elements are either disjoint or meet at right angles, see (7.29) of \cite{ACT}.  In particular, given any $x\in\BB^4$, there can be at most $4$ hyperplanes in this collection that contain $x$.

\subsection{Explicit formula for period map}
\label{subsec;explicitperiod}
To get an explicit formula for the period map, let $\gamma = \set{\gamma_0, \ldots, \gamma_4}$ be a basis, as a $
\ZZ[\omega]$-module, for the dual space $H_3(T,\ZZ)$ of $H^3(T,\ZZ)$, for which   the dual  hermitian form is diagonal with diagonal entries $
(-1,+1,+1,+1,+1)$.   (This choice is equivalent to a  framing as in (3.9) of  \cite{ACT}.)   Let $\Phi = \Phi(F)$ be the 
nonzero element $\Omega/F^{4/3}$  of $H^{2,1}_{\bar\omega}$ associated to the cubic form $F$ as in (3) of Theorem (6.5) of \cite{ACT}.  Then the \emph{period 
vector} associated to these  choices is the vector 
\begin{equation}
\label{eq:periodvector}
    \PPP(F,\gamma) = \left(\int_{\gamma_0} \Phi, \ldots, \int_{\gamma_4} \Phi
\right).
\end{equation}
More explicitly, let $f(x, y, z) = 0$ be an affine equation for the cubic surface
in one of the standard affine open sets of $\PP^3$.  Then
$w^3 = f(x,y,z)$ is an affine equation for the cubic threefold, 
and the  natural generator $\Phi(F)$ for $H^{2,1}_{\bar\omega}(T)$ is simply 
\[
   \Phi = \frac{dx\wedge dy \wedge dz}{w^4}.
\]
This is the analogue of the abelian differential $dx/y$ for an elliptic curve. 

\subsection{The main theorem of \cite{ACT}}
\label{subsec:maintheorem}

\begin{itemize}
\renewcommand{\labelitemi}{$\circ$}
\parskip5pt
 \item The map $\PPP$ of (\ref{eq:periodvector}) descends to a map $\PPP:\MMM_{sm}\map\Gamma\backslash\BB^4$ of the quotient spaces, and extends to the map $\PPP:\MMM_{st}\map\Gamma\backslash\BB^4$  of (\ref{eq:periodmap}), which is an isomorphism.  
 
 \item Moreover, if $\MMM_{sm}$ is the moduli space of \emph{smooth} cubic surfaces, then its restriction to $\MMM_{sm}$ gives an isomorphism
\[
\PPP:\MMM_{sm}\to \Gamma\backslash (\BB^4 - \HHH) 
\] 

\item  Finally, if $\MMM_{nod}$ is the moduli space of \emph{nodal} cubic surfaces, then $\MMM_{st} = \MMM_{sm}\cup\MMM_{nod}$ and the restriction of the period map to $\MMM_{nod}$ gives an isomorphism
\[
\PPP:\MMM_{nod}\to\Gamma\backslash\HHH = \Gamma_v\backslash \BB^3(v^\bot) = \Gamma_0\backslash\BB^3
\]
where $v$ is any fixed element of $\EEE^{1,4}$ with $h_0^*(v) = -1$.  For example, if $v = (0,1,0,0,0)$ and  $\Gamma_v$ is its stabilizer in $\Gamma$, then $\Gamma_0$ is the group $PU(1,3,\EEE)$ acting on $\BB^3$.

\end{itemize}

\subsection{The abelian varieties}
\label{subsec:imageabelian}
We need one final result that is not explicitly stated in \cite{ACT} but  follows easily from that paper.   In fact, it is the underlying philosophy  of proofs of the results just quoted.

\begin{itemize}
\renewcommand{\labelitemi}{$\circ$}
\parskip5pt
\item A principally polarized abelian variety $A\in\AAA_5$ is in the image $\iota(\Gamma\backslash\BB^4) = \iota\PPP(\MMM_{st})$ if and only if it has an automorphism $\sigma$ of order $3$ of type $(4,1)$, meaning that the induced action on $H^{1,0}(A)$ has eigenvalue $\omega$ of multiplicity $4$ and $\bar\omega$ of multiplicity $1$.

\item Furthermore, if $A\in\iota\PPP(\MMM_{st})$ as above, then $A\in\iota\PPP(\MMM_{sm})$ if and only if $A$ is, in addition, irreducible.

\item Similarly, if $A\in\iota\PPP(\MMM_{st})$ as above, then $A\in\iota\PPP(\MMM_{nod})$ if and only if $A$ is reducible and contains at least one irreducible factor isomorphic to a Fermat elliptic curve. 
 
\end{itemize}
The proof of the first statement follows from  the last statement of (\ref{subsec:intermediatejacobian}) and the fact that all automorphisms of type $(4,1)$ are conjugate under $Sp(10,\ZZ)$.  The second and third  statements follows from the fact that the intermediate Jacobian of a non-singular cubic threefold is irreducible, while each nodal surface gives rise to an $A_2$-singularity in the associated threefold, which is turn gives a summand of the Hodge structure of a Fermat elliptic curve.  See Lemmas (5.4), (5.7) of \cite{ACT} for the last point. 

\subsection{New results}
\label{subsec:newresults} 
Given this background, we now come to the results of this paper.  If $f$ has coefficients in a field $K$, then so does $\Phi$ --- it is
a $K$-rational differential. Consequently
$\Phi$ is well-defined up to a nonzero element of $K$.  Therefore questions
about the rationality of periods, referred to such a $K$-rational differential,
make sense for any field $L$ extending $K$.

There is a philosophy that periods of $K$-rational differentials are almost always transcendental;
this is, however, difficult to prove in any concrete instance.  Thus our focus will therefore be on \emph{period ratios}
such as $\omega_2/\omega_1$ for cubic curves or
\[
    v_i = \frac{\int_{\gamma_i} \Phi}{\int_{\gamma_0} \Phi}
\]
for cubic surfaces.  In the last ratio, we may rescale $\Phi$ so that $v_0 = 1$.
The resulting periods $v_i$, which are now relative to a differential with unknown rationality
properties, should be thought of as period ratios.  This will often be our point of view
in what follows.

Let us write the period vector as
$v = (a, b_1, b_2, b_3, b_4) = (a,b)$.  When $a = 1$,
we say that the period vector is \emph{normalized}.
In this case, since $(a,b)$ is a positive vector vector for $h_0^*$, we have $|a|^2 - |b|^2 >0$, so $b$ is a vector in $\CC^4$ of length at most 1.   
The corresponding Hodge structure of level
one and genus five is determined by  a $5\times10$ matrix
of periods $P = (A,B)$, where $P$ represents $\iota(a,b)$.   In the usual terminology,  $P$ is normalized in the case that $A$ is the identity
(in which case $B\in\HH_5$, that is, $B$ is symmetric with  positive-definite imaginary part).
Section (\ref{subsec:imageabelian}) tells us  that, in the presence of the action by $\sigma$, 
the period vector and the period matrix determine each other, up 
to natural equivalences (action of $\Gamma$ and $Sp(10,\ZZ)$).

Finally, recall that an abelian variety over $\CC$ is said to be 
of CM-type if (a) it is isogeneous to a product $A_1\times \cdots 
\times A_r$ of simple abelian varieties and (b) there are fields $K_i 
\subset End(A_i)\otimes\QQ$ such that $[K_i:\QQ] \ge 2\dim A_i$ 
(in which case $[K_i:\QQ]= 2\dim A_i$ and $K_i = End(A_i)\otimes
\QQ$.)  If the fields $K_i$ are equal, we say that $K = K_i$ is the 
CM field of the abelian variety.  See Mumford \cite[p. 347]
{Mumford:Shimura}.  We can now state our first result.

\begin{theorem}  Let $S$ be a cubic surface and let $J$ be its abelian 
\label{theorem:pvqomega}
variety.  
The following are 
equivalent: (a) one (and hence all) normalized period vectors of $S$ 
have coefficients in $\QQ(\omega)$;  (b) one (and hence all) 
normalized period matrices
of $S$ have coefficients 
in $\QQ(\omega)$, (c) $J$ is isogeneous to a product of 
Fermat elliptic curves.
\end{theorem}

It follows from the theorem that for a cubic surface with period 
vector in $\QQ(\omega)$, $End(J)\otimes_\ZZ \QQ$ is the ring of 
$5\times5$ matrices with coefficients in  $\QQ(\omega)$; as a 
consequence of the main theorem, we see that cubic surfaces with 
period vector in $\QQ(\omega)$ are of CM-type with CM field $
\QQ(\omega)$.  In particular, we get the following well-known fact:

\begin{corollary} The Hodge structures of CM type in $\BB^4$ are 
dense.
\end{corollary}

Because the period map $(\ref{eq:periodmap})$ is surjective, there is 
a dense set of smooth cubic surfaces with periods in $
\QQ(\omega)$.  As noted in  \cite[Theorem 11.6, 11.9]{ACT}, some explicit surfaces 
with periods of this kind are known. A period vector (not normalized) of the Fermat 
cubic surface $x^3 + y^3 + z^3 + w^3 = 0$ is 
\begin{equation}
\label{eq:periodfermat}
   v = (2 - \bar\omega, 1, 1, 1, 1).
\end{equation}
and that of the diagonal cubic surface $x^3 + y^3 + z^3 + w^3 + u^3 = 
0, x + y + 
x + w  + u = 0$ is
\begin{equation}
\label{eq:periodclebsch}
    v = (3,1,1,1,1).
\end{equation}
Again, it is more convenient to write this vector in non-normalized form.
The rationality properties of the differential form which gives the periods
are unknown.  Consequently the periods should be thought of as period ratios.

It is natural to ask how to characterize cubic surfaces with periods 
in $\QQ(\omega)$ in purely  geometric terms.  We do not know the answer, 
even conjecturally.

In a different direction, we state our 
 last result, which  concerns the opposite extreme of complex multiplication, namely cubic surfaces whose CM algebra is a field:  

\begin{theorem}  There exist simple, principally polarized abelian varieties $A$  of dimension five with three-fold symmetry  $\sigma$ of type $(4,1)$ and with rational endomorphism  ring $End(A)\otimes\QQ$  isomorphic to a CM-field $K$ of the form $K_0(\sqrt{-3})$, where $K_0$ is a totally real field of degree five.  The Hodge structure of $A$ is rational over $K$, and it is defined by a period vector rational over the same field.  By (\ref{subsec:imageabelian}), $A$ is the abelian varietly of a smooth cubic surface.
\end{theorem}

\noindent
{\bf Remark.} We study extreme cases of CM types.  In \cite[6.2]{Achter},  Achter 
classifies, up to isogeny, all possible CM types, endomorphism algebras, and Mumford-Tate groups of  abelian five-folds which admit an action of $\ZZ[\omega]$.

\section{Cubic surfaces with period vector rational over $\QQ(\omega)$}
\label{sec:pervecmat}

The key to the proof of the first theorem is an analysis of the 
periods of the most singular stable cubic surface.  We begin with 
the following:

\begin{proposition}  
\label{prop:cayley}
Any normalized period vector of the Cayley cubic 
surface,
\[
   \frac{1}{x} +  \frac{1}{y} +  \frac{1}{z} +  \frac{1}{w} = 0
\]
is $\Gamma$-equivalent to $v = (1,0,0,0,0)$.
\end{proposition}

\begin{proof}

The Cayley cubic surface is a cubic surface with $4$ nodes.  It  is classically known to be  the surface with the maximum 
possible number of nodes.   Thus, by the last part of (\ref{subsec:maintheorem}), it lies in the intersection of $4$ distinct hyperplanes in $\HHH$.  One such point in $(1,0,0,0,0)$, which is  the intersection of $(0,1,0,0,0)^\bot,\dots,(0,0,0,0,1)^\bot$.  It is easy to see that these are precisely the points in $v\in \EEE^{1,4}$ with $h_0^*(v) = 1$ and that all such points are equivalent under $\Gamma$.   Thus the period of the Cayley cubic is the $\Gamma$-orbit of $(1,0,0,0,0)$.
\end{proof}

Note that this argument implies the well-know classical fact that all $4$-nodal cubic surfaces are  isomorphic over 
the complex numbers.  Also (\ref{subsec:maintheorem}) combined with (\ref{subsec:hyperplanes}) reproves the classical fact quoted in this proof that $4$ is the maximum number of nodes.

 As we show in a moment, the formula we will derive for the period matrix in terms of the 
 period vector will give:
\begin{corollary}
\label{cor:matrixcayley}
Any normalized period matrix of the Cayley cubic surface is 
$\Gamma$-equivalent to
\begin{equation}
\label{eq:cayleyperiodmatrix}
   P = \left( I, \omega I \right).
\end{equation}
where $I$ is the identity matrix.Thus the intermediate Jacobian of 
the Cayley cubic surface is the product of five Fermat elliptic 
curves.\end{corollary}


\noindent
As a general principle, the period vector in our context determines 
the period 
matrix
and conversely:

Let us now compute the period matrix $P$ associated to a general 
cubic  surface $S$ with equation $F=0$, that is, the period matrix of the intermediate Jacobian of the associated cyclic cubic threefold $T$ as in (\ref{subsec:notation}). To this end, recall from (\ref{subsec;explicitperiod}) that 
$\gamma = \set{\gamma_0, \cdots, \gamma_4}$ 
is a unitary basis of $H_3(T,\ZZ)$ as a hermitian 
Eisenstein module.  Let 
$\gamma'= \set{\gamma_0', \cdots, \gamma_4'}
= \set{\sigma^{-1}\gamma_0, \sigma\gamma_1,\cdots, 
\sigma\gamma_4}$.  The the homology basis $\gamma \cup \gamma'$ is a sympletic 
basis with respect to the standard form
\[
J = \left(\begin{array}{rr}
     0 & I \\
     -I & 0
     \end{array}
     \right),
\]
where $I$ is the identity matrix.
Choose a basis $\set{\Phi^0, \ldots, \Phi^4}$ for $H^{2,1}$ where 
$\Phi^0 \in H^{2,1}_{\bar\omega}$ and $\Phi^i \in H^{2,1}_{\omega}$
for $i > 0$.  Then the period matrix
takes the form
\begin{equation}
\label{eq:pmatrix}
  P = \left( A,  B \right)
\end{equation}
where
\[
   A_{ij} = \int_{\gamma_j}\Phi^i \qquad 
   B_{ij} = \int_{\gamma_j'}\Phi^i.
\]
The change of variable formula in the calculus coupled with
the fact that the $\Phi^i$ are eigenvectors of $\sigma$ imply that
\begin{equation}
\label{eq:lambda}
  \int_{\gamma_j'}\Phi^i = \lambda \int_{\gamma_j}\Phi^i,
\end{equation}
where $\lambda \in \set{ \omega, \bar\omega}$.
We may choose the basis $\set{\Phi^1, \ldots, \Phi^4}$ so that $A_{ij} 
= 
\delta_{ij}$
for $i, j \in \set{1,\ldots, 4}$.  The first Riemann bilinear relation 
determines the first column of $A$ in terms of the second, while 
$(\ref{eq:lambda})$ determines $B$ in terms of $A$.  We conclude that
\begin{equation}
\label{eq:periodmatrix}
P =
\left(
\begin{array}{cccccccccc}
1 & b_1 & b_2 & b_3 & b_4 & \omega & \bar\omega b_1 & \bar\omega
b_2 & \bar\omega b_3 & \bar\omega b_4  \\
b_1 & 1 & 0 & 0 & 0 & \bar\omega b_1 & \omega  & 0  & 0 & 0 \\ 
b_2 & 0 & 1 & 0 & 0 & \bar\omega b_2 & 0  & \omega  & 0 & 0 \\
b_3 & 0 & 0 & 1 & 0 & \bar\omega b_3 & 0  & 0  & \omega & 0 \\
b_4 & 0 & 0 & 0 & 1 & \bar\omega b_4 & 0  & 0  & 0 & \omega
\end{array}
\right)
\end{equation}
 In the case of the Cayley cubic 
surface, Proposition (\ref{prop:cayley}) gives that the 
parameter vector $b$ is zero and the period matrix is
\[
P = \left( I, \omega I\right),
\]
where $I$ is the identity matrix.
It follows that the abelian variety of the Cayley cubic surface is the 
product of five Fermat elliptic curves.  This proves Corollary~\ref{cor:matrixcayley}.

The map $\psi(b) = Z$, where $Z = BA^{-1}$ and $A,B$ are as in (\ref{eq:pmatrix}), 
gives a matrix formula for the  imbedding  $\iota$ of (\ref{subsec:intermediatejacobian}) of the unit ball in the Siegel
upper half space of genus five.  It can be written more concisely as
\[
  P(b)  = \left(
  \begin{matrix}
  1 & b & \omega & \bar\omega b \\
  \transpose b & I & \bar\omega\transpose b & \omega I
  \end{matrix}
  \right)
\]
To understand better the location of $Z(b)$ in the Siegel upper
half space, note
the quantity $\delta = \det A = 1 - b\cdot b$ is nonzero since $|b|^2 < 1$.
Thus, if $P = (A,B)$ is the period matrix, we can form
$A^{-1}(A,B) = (1,Z)$, where $Z$ is the normalized matrix of $B$-periods,
a symmetric matrix with positive definite imaginary part.  The inverse of
$A$ is given by the following  expression:
\[
A^{-1} = 
\delta^{-1}
\left[
\begin{array}{ccccc}
1 & - b_1 & -b_2 & -b_3 & -b_4 \\
-b_1 & \delta_1 & b_1b_2 & b_1b_3 & b_1b_4 \\
-b_2 & b_1b_2 & \delta_2 & b_2b_3 & b_2b_4 \\
-b_3 & b_1b_3 & b_2b_3 & \delta_3 & b_3b_4 \\
-b_4 & b_1b_4 & b_2b_4 & b_3b_4 & \delta_4\end{array}
\right],
\]
where $\delta = 1 - (b_1^2 + b_2^2 + b_3^2 + b_4^2)$ is the determinant
of $A$ and where $\delta_i = \delta + b_i^2$.  Then 
\begin{equation}
\label{eq:ZFORMULA}
 Z = A^{-1}B = 
\delta^{-1}
\left[
\begin{array}{ccccc}
\delta_0' & -\theta b_1 & -\theta b_2 & -\theta b_3 & -\theta b_4 \\
-\theta b_1 & \delta_1' & \theta b_1b_2 & \theta b_1b_3 & \theta b_1b_4 \\
-\theta b_2 & \theta b_1b_2 & \delta_2' & \theta b_2b_3 & \theta b_2b_4 \\
-\theta b_3 & \theta b_1b_3 & \theta b_2b_3 & \delta_3' & \theta b_3b_4 \\
-\theta b_4 & \theta b_1b_4 & \theta b_2b_4 & \theta b_3b_4 & \delta_4'
\end{array}
\right],
\end{equation}
where 
\[
  \delta_0' = 
  \omega - \bar\omega(b_1^2 + b_2^2 + b_3^2 + b_4^2),
\]
and 
\[
   \delta_i' = \omega\delta - \theta b_i^2.
\]
Re-writing the  matrix $Z$, we obtain the following proposition:

\begin{proposition}
\label{prop:formulaimbedding}
Let $b\in\BB^4$ be a a row vector (thus $|b|<1$),  let $\transpose b$ be the 
corresponding column vector, and let  $b\otimes b$ be the matrix
whose $ij$-th entry is $b_ib_j$.   Then  
\begin{equation}
    Z(b) = \omega I + (\theta/\delta)
    \left(
    \begin{matrix}
    b \cdot b & -b \\
    - \transpose b & b\otimes b
    \end{matrix}
    \right),
\end{equation}
is  a  function of $b$ that maps points of the unit ball to points of the Siegel upper half space,
with the origin of the ball mapped to the normalized period matrix of product of five Fermat elliptic curves.  It gives a matrix formula for the imbedding $\iota:\BB^4\map\HH_5$ defined in (\ref{subsec:intermediatejacobian}).
\end{proposition}

\medskip
\noindent
{\bf Proof of Theorem 1.}

The proof  proceeds as follows. (1) The isogeny class of an abelian variety with lattice $\Lambda\subset \CC^n$ is the isomorphism class of the embedding $\Lambda\otimes \QQ \subset \CC^n$ (isomorphism by complex linear maps of $\CC^n$).
(2) If $\Lambda_0$ is the lattice of $E^5$, where $E$ is the Fermat elliptic curve, 
then $\Lambda_0\otimes \QQ = \QQ(\omega)^5 \subset \CC^5$.
(3) If the period vector $b$ is in $\QQ(\omega)$, then formula $(\ref{eq:periodmatrix})$ for the period matrix 
shows the columns of the matrix, which give a basis for the lattice $\Lambda$,  have 
entries in $\QQ(\omega)$. ÊThus $\Lambda\otimes \QQ$ is isomorphic to $\Lambda_0\otimes Q$ as rational subspaces of $\CC^5$.

\section{Proof of the second theorem}


As we have just seen, one can distinguish certain points
in the ball quotient, e.g., those whose period vector is rational
over $\QQ(\omega)$.  In that case the Abelian variety
is isogeneous to a product of Fermat elliptic curves.  

We now seek special points in the ball quotient where the corresponding 
abelian variety is simple and where the rational endomorphism
ring is a number field.   To this end, consider first a totally real number 
field $K_0$ of degree $n$, and a purely imaginary quadratic extension
  which we may write as $K = K_0(\sqrt \delta)$ for some element $\delta$ in $K_0$.  The field $K$  is a \emph{$CM$ field}.  It has $2n$ distinct embeddings $\tau_i$ in the complex numbers.  A
\emph{CM type} for $K$ is a choice 
$\Phi = (\tau_1, \ldots, \tau_n)$, where $\tau_i \ne \bar \tau_j$ for any $i, j$. Let $\OOO_K$ be the ring of integers
in $K$.  Then one can form the complex torus $A(K,\Phi) = \CC^n/\Phi(\OOO_K)$.  Note that the ring $\OOO_K$ acts by endomorphisms on $A(K,\Phi)$, so that
\begin{equation}
    K \subset End(A(K,\Phi))\otimes \QQ.
\end{equation}
If $A(K,\Phi)$ is simple, then the rational endomorphism ring is a division
algebra of dimension at most $2n$ over $\QQ$.  But $\dim_\QQ K = 2n$, so
\[
    K = End(A(K,\Phi))\otimes \QQ.
\]
Thus the rational endomorphism ring of such a torus is the field $K$.

To polarize the torus $A(K,\Phi)$, we follow an argument of Mumford, \cite[page 212]{Mumford:Abelian}.  There he claims 
\begin{quote}
$(*)$ the existence of an element
$\alpha$ of $K$ such that $\tau_i(\alpha) = \sqrt{-1}\beta_i$, where the 
$\beta_i$ are positive reals.
\end{quote}
 Given such an element, the expression
\begin{equation}
\label{eq:polarization}
  H(x,y) = 2\sum_{i=1}^g \beta_i\tau_i(x)\overline{\tau_i(y)}
\end{equation}
defines a positive Hermitian form on $K$.  Its associated skew form is
\begin{equation}
\label{eq:polarization2}
   \Omega(x,y) = \Im H(x,y) 
      = - 2\Re \sum_{i=1}^g \tau_i(\alpha)\tau_i(x)\overline{\tau_i(y)}
      = - Tr_{K/\QQ}(\alpha x \bar y).
\end{equation}
Since $\alpha$ is an element of the field $K$, the trace form takes rational
values.  To ensure that the values of the form are integers, we choose
$\alpha$ to be in $\OOO_K^\vee$, the lattice dual to $\OOO_K$.

To show the existence of the element $\alpha$ in $(*)$, Mumford argues as follows.
Let $\delta$
be an element of $K_0$ such that $K = K_0(\sqrt{\delta})$.  
Since $K$ is a totally imaginary extension of the totally real field $K_0$, for each $i$ we
can write $\tau_i(\sqrt{\delta}) = \sqrt{-1}\gamma_i$, where $\gamma_i$ is a nonzero real
number.  One can also find an element $\eta \in K_0$ such that $\tau_i(\eta)$
and $\gamma_i$ have the same sign for all $i$.  Then $\alpha = \eta\sqrt{\delta}$ is such that
 $\tau_i(\alpha) = \sqrt{-1}\beta_i$, where $\beta_i$ is a positive
real number for all $i$.  To summarize, we have the following:

\begin{proposition} Given an element $\alpha \in \OOO_K^\vee$ 
satisfying $(*)$, the expression $(\ref{eq:polarization})$
 defines a polarization of $A(K,\Phi)$.
\end{proposition}

There remains the question of whether this polarization, which 
is determined by a suitable element $\alpha \in K$, can be chosen 
to be principal.  If $\alpha \in \OOO_K^\vee$ and in addtion,
\begin{equation}
\label{eq:ooo}
   \alpha\OOO_K = \OOO_K^\vee,
\end{equation}
then the form $\Omega(x,y)$ is unimodular.  (See \cite[p 195]{Neukirch}, where
the dual as a module is identified with the dual with respect to the trace pairing.)
The polarization $(\ref{eq:polarization2})$ is principal exactly when condition $(\ref{eq:ooo})$
holds; this in turn hold if and only if $\alpha^{-1}$ generates the 
\emph{different} of $K$ --- the fractional ideal $(\OOO_K^\vee)^{-1}$.
(Thus in the cases we consider, the different is a principal ideal.) To conclude, we 
have

\begin{proposition} The polarization defined by $\alpha \in \OOO_K^\vee$  is principal
if and only if $\alpha^{-1}$ generates the different of $K$.
\end{proposition}

In the case that $K = K_0(\sqrt{-3})$, one can restate the criterion $(*)$
in terms of the different of $K_0$. To this end we introduce the following notion.

\begin{definition}
Let $\beta$ be an element of a totally real field $K_0$ of degree
$n$ over $\QQ$. Let $\tau_i$, $i = 1..n$ be the imbeddings of $K_0$
in the real numbers. Let $\epsilon$ be an $n$-vector with entries $\pm 1$.
Then $\beta$ is $\epsilon$-positive if 
\[
  \epsilon_i\tau_i(\beta)> 0 \hbox{ for all $i$.}
\] 
\end{definition}

\noindent
In the case that $\epsilon_i = +1$ for all $i$,
$\epsilon$-positivity the same as total positivity.

\begin{proposition}
\label{prop:criterionbeta}
Fix a vector $\epsilon = (\epsilon_1, \ldots, \epsilon_n)$ with $\epsilon_i = \pm 1$.
Let $\Phi$ be a CM-type for $K$ which extends  a set of embeddings
$\tau_i:K_0 \map \RR$ to embeddings $\tau_i: K \map \CC$ such that
$\tau_i(\sqrt{-3}) = \epsilon_i\theta$.
Suppose that there is an $\epsilon$-positive element $\beta$ 
 of $K_0$ whose inverse generates the different of $K_0$.  Then  $\alpha = -\beta\theta^{-1}$ defines a principal polarization of $A(K,\Phi)$.
\end{proposition}

\begin{proof}
The element $\alpha = -\beta\theta^{-1}$ is the product of inverses of generators
of the different for  $K_0$ and $\QQ(\sqrt{-3})$.  By the multiplicativity
of the different, this product is an inverse of a generator of the different of $K$, that is, a generator 
of $\OOO_K^\vee$.  See \cite[Prop. 2.2, p. 195]{Neukirch}.
Now 
\[
  \tau_i(\alpha) 
    = -\tau_i(\beta)\tau_i(\theta^{-1}) 
    = \sqrt{-1}\left(\epsilon_i\tau_i(\beta)/\sqrt{3}\right).
\]  
The quantity in parenthesis is positive if $\epsilon_i\tau_i(\beta)$ is positive,  Thus both $(\ref{eq:ooo})$ and $(*)$ 
are satisfied.  
\end{proof}

There remains the question of whether the conditions of $\epsilon$-positivity
for a generator of the different of $K_0$ can ever be satisfied.  The next result 
sets forth a criterion for its satisfaction which can sometimes be verified by computation.  We will carry out such a computation in the next section.

\begin{proposition}
Let $\Phi$ and $\beta$ be as in the preceding proposition.  Suppose that there is a unit $\eta$ of $\OOO_{K_0}$ such that $\tau_i(\eta)$ has
the same sign as $\epsilon_i\tau_i(\beta)$.
Then the element $\beta\eta$ is $\epsilon$-positive.
\end{proposition}

\begin{proof}
If the 
vectors $(\tau_1(\eta), \ldots , \tau_n(\eta))$  and
$(\epsilon_i\tau_1(\beta), \ldots , \epsilon_n\tau_n(\beta))$,
lie in the same octant then the conditions of the previous proposition are satisfied.
\end{proof}

Let us return to the case of $K = K_0(\sqrt{-3})$, where
$K_0$ is a totally real quintic field, and consider the problem 
of whether there are choices so that $A(K,\Phi, \alpha)$
is a principally polarized abelian variety which has
an automorphism of order three, and such that the 
eigenvalues of that automorphism on the space
of abelian differentials are $\omega$ with multiplicity
four and $\bar\omega$ with multiplicity one.

To this end, let 
$\tau_1, \ldots, \tau_5$ be the imbeddings of $K_0$ in $\RR$.
Extend them to imbeddings of $K$ in $\CC$ by requiring 
\begin{equation}
\label{eq:cmtypew}
\hbox{$\tau_1(\sqrt{-3}) = -\theta$ and $\tau_i(\sqrt{-3}) = \theta$
for $i > 1$.} 
\end{equation} 
Then $(\tau_1, \ldots, ,\tau_5)$ is a CM type for $K$ with the
property that $\Phi(\frac{-1 + \sqrt{-3}}{2})$ acts on $\CC^5$ with eigenvalues
$(\bar\omega, \omega, \omega, \omega, \omega)$.  It follows
that
\[
   \dim H^{1,0}_{\bar\omega}(A(K,\Phi)) = 1.
\]

Now let us construct a field with required properties.
Let $\zeta = \exp(2\pi\sqrt{-1}/11)$ be a primitive $11$-th 
root of unity.  The 
totally real subfield of $\QQ(\zeta)$ is $K_0 = \QQ(\rho)$,
where 
\[
  \rho = \zeta + \zeta^{-1} 
    = 2\cos(2\pi/11)
    = 1.6825...
\]
It is the totally really quintic field
of smallest discriminant, namely $11^4 = 14641$.   It is not hard to see
that $\QQ(\zeta)$ is a CM field and that one can construct a simple
Abelian variety from it. See  \cite[p. 24]{Lang:CM}.  
What we need, however, is a simple abelian variety with a suitable
action of $\omega$.  To this end, we establish the following:

\begin{proposition}  Let $K = K_0(\sqrt{-3})$.
Let $\Phi$ be the CM type extended from a vector of embedding $(\tau_1, \ldots \tau_5)$
of $K_0$ as in equation $(\ref{eq:cmtypew})$.  Then $A(K,\Phi)$ is a simple, principally polarized abelian
variety.  It therefore corresponds to a smooth cubic surface.
\end{proposition}

\begin{proof}
 The class number of $K_0$ is one, so that all ideals
in $\OOO_{K_0}$ are principal.  The different of $\OOO_{K_0}$ 
is the ideal generated by 
\[
   \delta_0 = -4r^4 + r^3 + 14r^2 + 4r - 9,
\]
where $r = -\rho$ generates $K_0$.
This element is not totally positive.  
The group of units of $\OOO_{K_0}$ is isomorphic to $C_2\times \ZZ^4$.
Generators for the free abelian part of this group are
\[
   u_1, u_2, u_3, u_4 = r^4 - 3r^2 + 1, r^2 - r - 1, r - 1, r,
\]
so that an arbitrary unit has the form $\pm u_1^j u_2^k u_3^\ell u_4^m$.
One finds that element
\[
    \delta = u_4\delta_0,
\]
which also generates the different ideal, is 
$\epsilon$-positive for $\epsilon = (+, +, -, +, +)$.  
The element $\beta = \delta^{-1}$ is also $\epsilon$-positive,
and it is the element we use to define the principal polarization.
The embeddings
of $\delta$ in $\RR$ are
\begin{verbatim}
  21.7307463515808
   4.26952134163076 
  -1.91569396353523 
  14.0542888631537
   5.86113740717001
\end{verbatim}
The corresponding embeddings of the generator $r$ of $K$ 
over $\QQ$ are
\begin{verbatim}
  -1.68250706566236
  -0.830830026003773
   0.284629676546570
   1.30972146789057
   1.91898594722899
\end{verbatim}
The criterion for $\epsilon$-positivity is satisfied,
and the abelian variety $A(K,\Phi, \alpha)$ is principally polarized.


To conclude that $A(K,\Phi, \alpha)$ corresponds to a smooth cubic surface,
following \ref{subsec:imageabelian},
it only remains to show that it is irreducible.  It will be enough to show
that it is simple.  To that end, consider the Galois
group $G$ of $K/\QQ$.  It is a group of order ten.  The subgroup
$Gal(K_0/\QQ)$ is normal.  Since it is of order five, it is cyclic.
The subgroup $Gal(\QQ(\omega)/\QQ)$ is also normal, and since it is of
order two, it is cyclic.  A group of order ten with these
structural features must be cyclic.

The subfields of $K$ are the fixed sets of subgroups of the Galois group.
Since a cyclic group of order ten has only two non-trivial subgroups,
the field $K$ has only two nontrivial subfields, namely, $K_0$ and $\QQ(\omega)$.
Once we know the subfields of $K$, Mumford's criterion for simplicity of $
A(K,\Phi, \alpha)$ \cite[p. 213-14]{Mumford:Abelian}
is easy to apply. If suffices to show that there are $\tau_i$ and $\tau_j$
such that $\tau_i|\QQ(\omega) \ne  \tau_j|\QQ(\omega)$.  The condition 
on the $\tau$'s holds by construction.  See equation $(\ref{eq:cmtypew})$.
\end{proof}

\noindent
We have a construction
\[
   \set{\hbox{Certain totally real quintic number fields}}
   \map
   \set{B^4/\Gamma}
\]
The period matrices that arise in this way are rational over $K$. Indeed, consider
the formula $(\ref{eq:ZFORMULA})$ for the normalized period
matrix $Z = (Z_{ij})$, where $i$ and $j$ run from 0 to 4.
Then the period vector is
\[
    b = (1/\theta)(Z_{01}, Z_{02}, Z_{03}, Z_{04})
\]
If $Z_{ij}$ is rational over $K$ then so is $b$.  The converse also comes
from formula $(\ref{eq:ZFORMULA})$.  The period vector
could be rational over $K_0$ and still correspond to a 
period matrix rational over $K$; it cannot, however, be rational over 
$\QQ$ or $\QQ(\omega)$.

\section{Computations and experiments}

Above we described a method to show that the totally real quintic
field of discriminant 14641
defines a principally polarized abelian variety with three-fold symmetry
of the correct type.  The same method can be used to produce lists
of quintic fields with this property.  Using the Sage code below,
for example, we show that (1) there are 414 totally real quintic fields
of discriminant less than $10^6$; (2) none of these have class number bigger
than 1; (3) there are 412 fields which satisfy the hypotheses
of Proposition \ref{prop:criterionbeta} (4) the density of such fields in the indicated range
is about 0.995.  Below is the data for the fields of discriminant $< 10^5$.
The second column is the discriminant of the field.  The third column 
answers the question: \emph{Is there a cubic surface with the given discriminant?}
\begin{verbatim}
        discr   	result
    1   14641    True
    2   24217    True
    3   36497    True
    4   38569    True
    5   65657    True
    6   70601    True
    7   81509    True
    8   81589    True
    9   89417    True
\end{verbatim}

\noindent
{\bf Sage code.}  We first define a function to return the class number of the field
$K_0 = \QQ[X]/(F[X])$ for a polynomial $F$:
\begin{verbatim}
   def classNumber(F):
     R.<x> = PolynomialRing(QQ)
     f = R(F)
     K.<a> = NumberField(f)
     return K.class_number()
\end{verbatim}
\noindent
Next, we define the function {\tt test(F)}.  It returns
{\tt True} if and only if there is an $\epsilon$-positive
generator of the different of $K_0$ for some $\epsilon$ with a single $-1$.  
In paragraph one of the code,
a generator {\tt d0} of the different  and generators {\tt u[0], ..., u[3]}
for the free part of the group of units are found for $K_0$.

In the second paragraph, a search is conducted over a the set
$\set{0,1}^5$ in $i,j,k,\ell, m$ space.  For each element of the set,
we ask whether $d = (-1)^id_0 u_0^j  u_1^k  u_2^\ell  u_3^m $ is 
$\epsilon$-positive. If a lattice point passes the test, the function
{\tt test} returns {\tt True}.  If no lattice point passes the test,  it returns {\tt False}.  
 The validity of the test rests on the fact that 
 the positivity properties of  $\tau_r(d)$ depend only on the exponents 
 $j,k,\ell, m$ modulo 2.

\begin{verbatim}
  def test(F):
    R.<x> = PolynomialRing(QQ)
    f = R(F)
    K0.<a> = NumberField(f);
    D0 = K0.different(); d0 = D0.gens_reduced()[0]
    u = K0.units()
  
    i, j, k, l = 0, 0, 0, 0
    count = 0
    for i in range(0,2):
      for j in range(0,2):
        for k in range(0,2):
          for l in range(0,2):
            for m in range(0,2):
              sign = (-1)^i
              dd = sign*d0*u[0]^j*u[1]^k*u[2]^l*u[3]^m
              if epsilon_positive(dd) == True:
                return True
    return False
\end{verbatim}

\noindent
The test for $\epsilon$-positivity is carried out by the function below:

\begin{verbatim}
  def  epsilon_positive(x):
    pos_minus_neg = sum( sgn(tx) for tx in x.complex_embeddings() )
    return ( pos_minus_neg == 3 )
\end{verbatim}

\noindent
Finally, we enumerate the totally real quintic fields of discriminant
less than N, applying the above test to each, and collecting
various statistics.
\goodbreak
\begin{verbatim}
  def testFields(N):
  
    TRF = enumerate_totallyreal_fields_all(5, N)
    print "Number of totally real fields:", len(TRF)
  
    n = 1
    nNonUFD = 0
    nPass = 0
    ratio = 0;
  
    for field in TRF:
      discr, G = field
      cn = classNumber(G)
      result = False
      if (cn == 1):
        result = test(G)
        if (result):
          nPass = nPass + 1
      else:
        nNonUFD = nNonUFD + 1
      ratio = nPass/(1.0*n)
      print "%5d %4d %s" % (n, discr, result)
      n = n + 1
  
    print "Summary:"
    print "  Number of fields:", len(TRF)
    print "  Number of fields of class number > 1:", nNonUFD
    print "  Number of fields which satisfy the criterion:", nPass
    print "  Ratio:", ratio
\end{verbatim}

\noindent
To run the test on fields of discriminant $< 10^6$,
run the command {\tt testFields(10\^{}§6)}.

\section{Problems}

As promised in the introduction, we list some problems suggested by the analogy between the period map $\PPP$ for cubic surfaces and the classical period map for cubic curves.

\begin{itemize}
\renewcommand{\labelitemi}{$\circ$}
\parskip5pt
\item {\it  Find explicit values of the period map}.  This can be interpreted in two ways:  the actual period vector  $\PPP(F,\gamma)\in \CC^{1,4}$ as in (\ref{subsec;explicitperiod}), or as a period ratio, namely the equivalence class of $\PPP(F,\gamma)$ in $\BB^4$, or in  $\Gamma\backslash\BB^4$.  We remark that nothing seems to be known about the periods $\PPP(F,\gamma)$.  The known period ratios from \S 11 of \cite{ACT},   (\ref{eq:periodfermat}) and (\ref{eq:periodclebsch}) above, are found by symmetry considerations.   The Fermat and Clebsch surfaces are uniquely determined by their automorphism group \emph{and} this group is generated by reflections.   This suggests an  easier sub-problem:

\item {\it Find values of the period map for all surfaces with symmetries}  These surfaces are totally classified.  One special family of such surfaces that should be more tractable is the following:

\item{\it Find the values of the period map for all cyclic cubic surfaces.  Express them as explicit functions of the periods of the corresponding cubic curves.}  A cyclic cubic surface means a surface $S\subset\PP^3$ given by an equation $x_3^3 = f(x_0,x_1,x_2)$ where $f(x_0,x_1,x_2)$ is a non-singular cubic form, in other words, $S$ is the cyclic cover of $\PP^2$ branched along the cubic curve $f=0$. 

\item {\it Study special values of the period map}.  We mean the situation when the abelian variety has complex multiplication.

\item {\it Find an explicit inverse to the period map}.   We mean: given $b\in\BB^4-\HHH$, find an explicit cubic form $F$ so that the class of $\PPP(F,\gamma)$ is $b$.   In principle this can be done using suitable theta functions.   Allcock and Freitag \cite{AF} and Matsumoto and Terasoma \cite{matsumoto} have given projective embeddings of $\Gamma\backslash\BB^4$ with image an algebraic embedding of (a covering of)  $\MMM_{st}$.  In this way one obtains formulas for algebraic invariants of $F$ in terms of $b$. but we are not aware of explicit formulas for a cubic form.

\item {\it Study the algebraic and arithmetic nature of the period map}.  Both the domain and target of $\PPP:\MMM_{st}\map\Gamma\backslash\BB^4$ are algebraic varieties and $\PPP$, in spite of its transcendental definition, is an algebraic map.   See \cite{Achter} for more information and for progress on this question.

\end{itemize}

\vskip0.30truein
\obeylines
\parskip=0pt
James A. Carlson: jcarlson AT claymath.org
Clay Mathematics Institute
\bigskip
Domingo Toledo: toledo AT math.utah.edu
Department of Mathematics
University of Utah

\end{document}